\documentclass[12pt]{amsart}
\usepackage[all]{xy}
\usepackage{amssymb}
\usepackage{amsthm}
\usepackage{hyperref}
\usepackage{todonotes}
\hypersetup{colorlinks=true,linkcolor=blue,citecolor=magenta}
\usepackage{amsmath}
\usepackage{mathtools}
\usepackage{amscd,enumerate}
\usepackage{verbatim}
\usepackage{eurosym}
\usepackage{float}
\usepackage{color}
\usepackage{cases}
\usepackage{bm}
\usepackage{dcolumn}
\usepackage[mathscr]{eucal}
\usepackage{bbm}
\usepackage{comment}
\usepackage{physics}
\usepackage{adjustbox}
\usepackage[textheight=8.77in, textwidth=6.8in]{geometry}

\usepackage{breakurl}
\usepackage{caption, tabu, multicol}
\usepackage[autostyle, english = american]{csquotes}
\MakeOuterQuote{"}

\usepackage[]{breqn}

\newtheorem*{thm*}{Theorem}
\newtheorem*{conj*}{Conjecture}

\newtheorem*{remark}{Remark}

\newtheorem{theorem}{Theorem}[section]
\newtheorem{cor}[theorem]{Corollary}

\newtheorem{lemma}[theorem]{Lemma}
\newtheorem{prop}[theorem]{Proposition}
\newtheorem{proposition}[theorem]{Proposition}

\newtheorem{example}{Example}[section]
\newtheorem*{example*}{Example}

\newtheorem*{TraceFormula}{Eichler-Selberg Trace Formula}

\newcommand{\Z}{\mathbb{Z}}

\newcommand{\Q}{\mathbb{Q}}
\newcommand{\F}{\mathbb{F}}

\newcommand{\C}{\mathbb{C}}
\newcommand{\SL}{\operatorname{SL}}
\renewcommand{\Tr}{\operatorname{Tr}}

\newcommand{\ep}{\varepsilon}

\makeatletter
\DeclarePairedDelimiterX{\pmodx}[1]{(}{)}{{\operator@font mod}\mkern6mu#1}
\renewcommand{\pmod}{%
  \allowbreak
  \if@display\mkern18mu\else\mkern8mu\fi
  \pmodx
}
\makeatother

\numberwithin{equation}{section}

\begin{document}
\title{Supersingular loci from traces of Hecke Operators}

\author[K. Gomez, K. Lakein, A. Larsen]{Kevin Gomez, Kaya Lakein, and Anne Larsen}
\address{Department of Mathematics, Vanderbilt University, Nashville, TN 37235}
\email[K. Gomez]{kevin.j.gomez@vanderbilt.edu}
\address{Department of Mathematics, Stanford University, Stanford, CA 94305}
\email[K. Lakein]{epi2@stanford.edu}
\address{Department of Mathematics, Harvard University, Cambridge, MA 02138}
\email[A. Larsen]{larsen@college.harvard.edu}

\begin{abstract}
A classical observation of Deligne shows that, for any prime $p \geq 5$, the divisor polynomial of the Eisenstein series $E_{p-1}(z)$ mod $p$ is closely related to the supersingular polynomial at $p$,
\begin{align*}
    S_p(x) := \prod_{E/\overline{\F}_p \text{ supersingular}}(x-j(E))\,\,  \in \F_p[x].
\end{align*}
Deuring, Hasse, and Kaneko and Zagier found other families of modular forms which also give the supersingular polynomial at $p$.
In a new approach, we prove an analogue of Deligne's result for the Hecke trace forms $T_k(z)$ defined by the Hecke action on the space of cusp forms $S_k$. We use the Eichler-Selberg trace formula to identify congruences between trace forms of different weights mod $p$, and then relate their divisor polynomials to $S_p(x)$ using Deligne's observation.
\end{abstract}

\maketitle

\section{Introduction and Statement of Results}\label{sec: intro}
Let $K$ be a field of characteristic $p > 0$, and let
$\overline{K}$ be its algebraic closure. An elliptic curve $E/K$ is called \textit{supersingular} if the group $E(\overline{K})$ has no $p$-torsion. This condition depends only on the $j$-invariant of $E$, and there are only finitely many supersingular $j$-invariants in $\overline{\F}_p$ \cite{K-Z}. 
Thus, the \textit{supersingular locus at p} can be identified with the polynomial
\begin{equation}
    S_p(x) := \prod_{E/\overline{\F}_p \text{ supersingular}}(x-j(E)).
\end{equation}
If $E/\overline{\F}_p$ is supersingular, then $j(E) \in \F_{p^2}$ \cite[Theorem 3.1]{Silverman}, hence $S_p(x)$ splits completely in $\F_{p^2}$. Moreover, the Galois conjugate of any $j$-invariant of a supersingular elliptic curve over $\overline{\F}_p$ will be the $j$-invariant of another supersingular elliptic curve over $\overline{\F}_p$, therefore $S_p(x) \in \F_p[x]$ factors into linear and quadratic irreducibles in $\F_p[x]$. The objective of this paper is to determine $S_p(x)$ using the theory of modular forms, in particular the theory of Hecke operators on spaces of integral weight cusp forms for $\SL_2(\Z)$.

Besides $S_p(x)$, it will turn out to be convenient to define the closely related polynomial
\begin{equation}
    \widetilde{S}_p(x) := \prod_{\substack{E/\overline{\F}_p \text{ supersingular} \\ j(E) \neq 0,1728}} (x-j(E)).
\end{equation}
If $p \geq 5$, then $\widetilde{S}_p(x)$ has degree $\lfloor \frac{p}{12}\rfloor$, and
$S_p(x) = x^{\delta_p} (x-1728)^{\ep_p} \widetilde{S}_p(x)$ \cite[Proposition 2.49]{OnoCBMS}, where
\begin{equation}\label{def: deltaepsilon}
    \delta_p := \begin{dcases} 0 & \text{if } p \equiv 1 \pmod 3, \\ 1 & \text{if } p \equiv 2 \pmod 3, \end{dcases} \quad \text{ and } \quad \ep_p := \begin{dcases} 0 & \text{if } p \equiv 1 \pmod 4, \\ 1 & \text{if } p \equiv 3 \pmod 4. \end{dcases}
\end{equation}
\newpage
For instance, the factorizations of $S_p(x)$, $\widetilde{S}_p(x)$ in $\F_p[x]$ for several primes $p \geq 5$ are as follows:
\begin{table}[H]
\centering
\begin{tabu}{|c|l|l|}
    \hline
    $p$ & $S_p(x)$ & $\widetilde{S}_p(x)$ \\ \hline
    $5$ & $x$ & $1$ \\ \hline
    $7$ & $x+1$ & $1$ \\ \hline
    $13$ & $x+8$ & $x+8$ \\ \hline
    $17$ & $x(x+9)$ & $x+9$ \\ \hline
    $29$ & $x(x+4)(x+27)$ & $(x+4)(x+27)$ \\ \hline
    $37$ & $(x+29)(x^2+31x+31)$ & $(x+29)(x^2+31x+31)$ \\ \hline
\end{tabu}
\vspace{2mm}
\caption{Supersingular loci for small primes $p$}
\vspace{-6mm}
\end{table}

For any even integer $k \geq 4$, we denote the space of modular forms of weight $k$ on $\SL_2(\Z)$ by $M_k$, and the subspace of cusp forms by $S_k$. Any such $k$ can be written uniquely in the form $k = 12m+4\delta+6\ep$ with $m \in \Z_{\geq 0}$, $\delta \in \{0,1,2\}$, and $\ep \in \{0,1\}$. Let $\Delta(z) \in S_{12}$ denote the usual discriminant function, and let $E_4(z)$ and $E_6(z)$ denote the Eisenstein series of weights $4$ and $6$, respectively. Then if $f \in M_k$, we have
\begin{equation}\label{eq: fdecomposition}
    f(z) = \Delta(z)^mE_4(z)^\delta E_6(z)^\ep F(f;j(z))
\end{equation} 
for some polynomial $F(f;x) \in \C[x]$ of degree $m$, which has leading coefficient equal to the constant term of the Fourier expansion of $f(z)$ \cite{K-Z}. We call $F(f;x)$ the \textit{divisor polynomial} of $f$.

A classical observation of Deligne
\cite{Serre}
establishes that
\begin{equation}\label{eq: Deligne}
    S_p(x) \equiv x^{\delta_p} (x-1728)^{\ep_p} F(E_{p-1}; x) \pmod p.
\end{equation}
Combined with \eqref{eq: fdecomposition} and the fact that $E_{p-1}(z) \equiv 1 \pmod p$ for any prime $p \geq 5$ \cite[Lemma 1.22]{OnoCBMS}, this relation yields an explicit expression for $S_p(x)$.
The Eisenstein series $E_{p-1}(z)$ is, however, not the only modular form that can be used to obtain information about the supersingular locus at $p$.
Deuring and Hasse \cite{Serre} proved the analogue of \eqref{eq: Deligne} for the modular forms $H_{p-1} \in M_{p-1}$, defined as the coefficient of $x^{p-1}$ in the polynomial
\begin{equation*}
    (1-3E_4(z)x^4+2E_6(z)x^6)^{\frac{p-1}{2}} = 1-\frac{3(p-1)}{2}E_4(z)x^4+(p-1)E_6(z)x^6+O(x^8),
\end{equation*}
and Kaneko and Zagier \cite[Theorem 1]{K-Z} showed the same relation for two further families of modular forms $G_{p-1}$ and $F_{p-1}$, defined as the coefficient of $x^{p-1}$ in the expansion
\begin{equation*}
    (1-3E_4(z)x^4+2E_6(z)x^6)^{-\frac 12} = 1+\frac{3}{2}E_4(z)x^4-E_6(z)x^6+\frac{27}{8}E_4(z)^2x^8+O(x^{10}),
\end{equation*}
and the unique normalized solution in $M_{p-1}$ to a particular differential equation, respectively.
\begin{remark}
Along different lines, Atkin defined a measure on $\Q[j(z)]$ and a 
family of orthogonal polynomials whose reductions mod $p$ recover 
$S_p(x)$ \cite[Theorem 3]{K-Z}. Moreover, Ahlgren and Ono showed that if $\mathrm{wt}(Q)$ is the Weierstrass weight of $Q \in X_0(p)$ and $g_p$ is the genus of $X_0(p)$, 
then
\begin{equation*}
    \prod_{Q \in X_0(p)} (x-j(Q))^{\mathrm{wt}(Q)} \equiv S_p(x)^{g_p(g_p-1)} \pmod p,
\end{equation*}
hereby demonstrating that the arithmetic of Weierstrass points on the modular curve $X_0(p)$ offers another way of locating $S_p(x)$ \cite[Theorem 1]{Ahlgren}.
\end{remark}

Given these results, it is natural to ask whether there are other canonical families of modular forms whose reductions mod $p$ encode information about $S_p(x)$. For example, given that $M_k = \C E_k \oplus S_k$, one might ask whether any particular element of $S_k$ also has a divisor polynomial divisible by the supersingular $j$-invariants corresponding to some prime $p \geq 5$. For non-negative even integers $k \neq 2$, we consider the basis for $S_k$ consisting of simultaneous eigenforms of the family of Hecke operators $\{T(n): n \ge 2\}$ \cite[Definition 2.1]{OnoCBMS}. We can extract information about the supersingular locus from the Hecke trace forms $T_k(z)$, defined as the sum of the Hecke eigenforms in $S_k$ (each normalized to have leading coefficient 1), or equivalently,
\begin{equation}
T_k(z) := (\dim_\C S_k)\, q + \sum_{n=2}^\infty \Tr_k(n) q^n,
\end{equation}
where $\Tr_k(n)$ is the trace of $T(n)$ acting on $S_k$.

\begin{theorem}\label{thm: congruence}
    If $p \geq 5$ is prime, $k \geq 4$ is an even integer, and $$ n = \begin{cases}
    (p^2+p)\lfloor \frac{k}{p^3-p} \rfloor & k \not\equiv 0,2 \pmod{p^3-p},\\
    (p^2+p)\big(\lfloor \frac{k}{p^3-p} \rfloor - 1\big) & k \equiv 0,2 \pmod{p^3-p},
    \end{cases}$$ 
    then
    $\widetilde{S}_p(x)^n$ divides the divisor polynomial $F(T_k;x)$ mod $p$.
\end{theorem}

This result will follow immediately from Theorem~\ref{thm: preccong}, which expresses $F(T_k;x)$ mod $p$ as a product of the factors $\widetilde{S}_p(x)$, $x$, $x-1728$, and the divisor polynomial of a lower weight trace form.
In certain special cases, we get the following simple factorization:
\begin{cor}\label{cor: sspmod}
    If $p \ge 5$ is a prime and $k \geq 4$ is an even integer such that $k \equiv 12,16,18,20,22,26 \pmod{p^3-p}$, then
    \begin{equation*}
        F(T_k;x) \equiv \widetilde{S}_p(x)^n x^{\alpha(p,k)} (x - 1728)^{\beta(p,k)} \pmod{p},
    \end{equation*}
    where $n = (p^2+p) \lfloor \frac{k}{p^3-p} \rfloor$, and $\alpha(p,k)$ and $\beta(p,k)$ are given in Corollary~\ref{cor: alphabeta}.
\end{cor}

\begin{example*}
Let $p = 13$ and $k = 2196$. Note that $k = p^3 - p + 12$, hence $k$ satisfies the condition of Corollary~\ref{cor: sspmod}. We have that
\begin{align*}
    T_{2196}(z) &= 183q + 930885\dots406856q^2 + \dots \\
   &\equiv q + 2q^2 + 5q^3 + 10q^4 + 7q^5 + \dots \pmod{13}
\end{align*}
where the coefficient of $q^2$ is $\,\approx 9.3 \times 10^{328}$. The divisor polynomial of $T_{2196}(z)$ is then
\begin{equation*}
    F(T_{2196}; x) \equiv (x + 8)^{182} \equiv \widetilde{S}_{13}(x)^n \pmod{13},
\end{equation*}
where $n = (13^2 + 13)\lfloor \frac{2196}{13^3 - 13} \rfloor = 182$, and $\alpha(13,2196) = \beta(13,2196) = 0$.
\end{example*}

Furthermore, we define the modified trace forms
\begin{equation}\label{eq: That}
    \widehat{T}_k^{(p)}(z) := \sum_{\substack{n \geq 1 \\ (n,p)=1}} \Tr_k(n) q^n \in \F_p[[q]],
\end{equation}
where for the purpose of defining the divisor polynomial $F(\widehat{T}_k^{(p)};x)$, we consider $\widehat{T}_k^{(p)}(z)$ to be the mod $p$ reduction of a modular form of weight $k+p^2-1$ \cite[Proposition 2.44]{OnoCBMS}. We then obtain the following analogue to Theorem~\ref{thm: congruence}, which will again follow from a more complete description of the divisors of $F(\widehat{T}_k^{(p)};x)$, given in Theorem~\ref{thm: prechatcong}.
\begin{theorem}\label{thm: hatcongruence}
   If $p \geq 5$ is prime, $k \geq p^2-1$ is an even integer such that $k \equiv 0,4,6,8,10,14 \pmod{p^2-1}$, and $n = (p+1)(\lfloor \frac{k}{p^2-1} \rfloor-1)$, then $\widetilde{S}_p(x)^n$ divides the divisor polynomial of $\widehat{T}_k^{(p)}(z)$.
\end{theorem}

\begin{example*}
    Let $p = 19$ and $k = 724$. Note that $k = 4 + 2(p^2 - 1)$, hence $k$ satisfies the conditions of Theorem~\ref{thm: hatcongruence}. We have that
    \begin{equation*}
        \widehat{T}^{(19)}_{724}(z) = 3q + 12q^4 + 10q^5 + 14q^7 + 8q^9 + \dots.
    \end{equation*}
    The divisor polynomial of $\widehat{T}^{(19)}_{724}(z)$ is such that
    \begin{equation*}
        F(\widehat{T}^{(19)}_{724}; x) \equiv 3 \cdot x^6 (x + 1)^{24} (x + 12)^{48} (x^2 + 7x + 5) Q_1(x) Q_2(x) \pmod{19},
    \end{equation*}
    where
    \begin{align*}
        Q_1(x) &\equiv x^4 + 11x^3 + 6x^2 + 17x + 13 \pmod{19}, \\
        Q_2(x) &\equiv x^5 + 9x^4 + 3x^3 + x^2 + 7x + 4 \pmod{19}.
    \end{align*}
    We then see that $\widetilde{S}_{19}(x) \equiv (x + 12) \pmod{19}$ divides $F(\widehat{T}^{(19)}_{724}; x)$ with multiplicity greater than or equal to $(19+1)(\lfloor \frac{724}{19^2 - 1} \rfloor - 1) = 20$.
\end{example*}

Given modular forms $f$ and $g$ of even weights $k_1$ and $k_2$, such that $k_2 > k_1 \geq 4$, which are equivalent (and nonzero) mod $p$, the theory of modular forms mod $p$ implies that $k_2 - k_1 = n(p-1)$ for some $n \in \Z^+$ \cite[Proposition 2.43]{OnoCBMS}. Thus, $g$ and $E_{p-1}^n f$ have the same divisor polynomial mod $p$, where it follows from Deligne's observation that the divisor polynomial of $E_{p-1}^n f$  mod $p$ is a product of factors of $\widetilde{S}_p(x), x, (x-1728)$, and the divisor polynomial of $f$. To prove Theorems \ref{thm: congruence} and \ref{thm: hatcongruence}, we thus find equivalences between trace forms with different weights using the Eichler-Selberg trace formula.

The paper is organized as follows: In Section~\ref{sec: congruences}, we recall the Eichler-Selberg trace formula and the Kronecker-Hurwitz class number relation, and prove Theorems~\ref{thm: traceclassification} and \ref{thm: hatclassification}, which describe pairs of weights $k_2 > k_1 \geq 4$ whose trace forms are (up to a constant multiple) equivalent mod $p \ge 5$.
In Section~\ref{sec: modformsmodp}, we derive a formula (see Proposition~\ref{prop: divpolycong}) relating the divisor polynomials of modular forms with different weights which are equivalent mod $p$ using Deligne's observation that $F(E_{p-1};x) \equiv \widetilde{S}_p(x) \pmod{p}$.
In Section~\ref{sec: mainproofs}, we combine Theorems~\ref{thm: traceclassification} and \ref{thm: hatclassification} with Proposition~\ref{prop: divpolycong} to prove Theorems~\ref{thm: preccong} and \ref{thm: prechatcong} describing the factorizations of the divisor polynomials of certain trace forms and modified trace forms, from which we immediately obtain Theorems~\ref{thm: congruence} and \ref{thm: hatcongruence}.
We conclude with some illustrative numerical examples of our main theorems in Section~\ref{sec: numerics}.

\section*{Acknowledgements}
We would like to thank Ken Ono for suggesting and advising this project, and for many helpful conversations and suggestions. We thank W. Craig and B. Pandey for their valuable comments. Finally, we are grateful for the generous support of the National Science Foundation (DMS 2002265 and DMS 205118), the National Security Agency (H98230-21-1-0059), the Thomas Jefferson Fund at the University of Virginia, and the Templeton World Charity Foundation.

\section{Congruences for Hecke Trace Forms}\label{sec: congruences}
The main results of this section, which will be used to prove Theorems~\ref{thm: congruence} and \ref{thm: hatcongruence} in Section~\ref{sec: mainproofs}, are the following two theorems regarding congruences between Hecke trace forms of different weights mod $p$:

\begin{theorem}\label{thm: traceclassification}
    If $p \geq 5$ is prime and $k_1,k_2$ are even integers such that $k_2 > k_1 \geq 4$, then
    \begin{equation}\label{eq: Tcongruence}
    T_{k_2}(z) \equiv m \cdot T_{k_1}(z) \pmod{p}
    \end{equation} for some $m \in \F_p^\times$ whenever $k_1$ and $k_2$ are related by one of the following conditions:
    \begin{enumerate}[(i)]
    \item $k_2 = k_1 + cp(p^2 - 1)$ for $c \in \Z^+$ (in which case $m=1$), 
    \item $k_2 = k_1 + c(p^2 - 1)$ for some $c\in\Z^+$ such that $c + 1 \not\equiv 0 \pmod p$, where $p \leq 11$ and $k_1 = (p^2 - 1) + 0,4,6,8,10,14$ (in which case $m=c+1$), or
    \item $k_2 = k_1 + c(p - 1)$ for $c \in \Z^+$, where  $p \leq 11$
    and $\dim S_{k_1} = \dim S_{k_2}$ (in which case $m=1$).
\end{enumerate}
\end{theorem}

\begin{remark}
For any prime $p \ge 5$, the theory of modular forms mod $p$ shows that congruences of the form \eqref{eq: Tcongruence} can only occur for pairs of even weights $k_2 > k_1 \geq 4$ that differ by a multiple of $p-1$ \cite[Proposition 2.43]{OnoCBMS}. Apart from this, it is natural to ask whether the conditions for such congruences given in Theorem \ref{thm: traceclassification} are comprehensive. This is the case for $p \in \{5,7,11\}$, and a general search of trace forms of weight up to 300 did not reveal any congruences except those predicted by Theorem \ref{thm: traceclassification} for $p \ge 13$.
\end{remark}

\begin{theorem}\label{thm: hatclassification}
    If $p \geq 5$ is prime, $k\in \{0,4,6,8,10,14\}$, and $m \in \Z_{\ge 1}$, then
    \begin{equation}\label{pminusonethetas}
    \widehat{T}_{k+m(p^2-1)}^{(p)}(z) = m \cdot \widehat{T}_{k+p^2-1}^{(p)}(z).
    \end{equation}
\end{theorem}

\begin{remark}\label{rem: theta}
Note that given any even integer $k \geq 4$, we have that $\widehat{T}_k(z)$ is the $q$-expansion obtained by applying $\Theta^{p-1}$ to $T_k(z)$ and reducing each coefficient mod $p$, where $\Theta = q\,\frac{d}{dq} = \frac{1}{2\pi i}\,\frac{d}{dz}$. Strictly speaking, we obtain a relation as in \eqref{pminusonethetas} by applying $\Theta$ to $T_{k+(p^2-1)}(z)$ and $T_{k+m(p^2-1)}(z)$ only once and reducing mod $p$. The effect of applying this operator $p-1$ times instead
is that the resulting $q$-expansion is identical to the original one, except that all coefficients whose order is divisible by $p$ are annihilated. 
\end{remark}

In Section~\ref{subsec: trace}, we recall the Eichler-Selberg trace formula, which we use to prove Theorem~\ref{thm: traceclassification}. The proofs of Theorems~\ref{thm: traceclassification} and \ref{thm: hatclassification} then follow in Sections~\ref{subsec: proofoftrace} and \ref{subsec: proofofhat}.

\subsection{The Eichler-Selberg Trace Formula}\label{subsec: trace}

Following Zagier's appendix in \cite{ZagierTraceFormula}, we define the Hurwitz class numbers $H(n)$ for $n \in \Z$ as follows: let $H(n) = 0$ whenever $n < 0$, and $H(0) = -\frac{1}{12}$. For $n > 0$, define $H(n)$ to be the number of equivalence classes with respect to $\SL_2(\Z)$ of positive definite binary quadratic forms $$f(x,y) = ax^2+bxy+cy^2, \quad a,b,c \in \Z,$$ with discriminant $\Delta(f) := b^2 - 4ac = -n$, where we count forms equivalent to a scalar multiple of $x^2+y^2$ with multiplicity $\frac 12$, and forms equivalent to a scalar multiple of $x^2+xy+y^2$ with multiplicity $\frac 13$.

Moreover, for any even integer $k > 0$, we define the polynomial $P_k(t,n)$ to be the coefficient of $x^{k-2}$ in the series expansion of $(1-tx+nx^2)^{-1}.$
Then we have that 
\begin{equation}\label{eq: P}
P_k(t,n) = \frac{\rho^{k-1}-\bar{\rho}^{k-1}}{\rho - \bar{\rho}}\,,
\end{equation}
where $\rho, \bar{\rho}$ are defined by the relations $\rho +\bar{\rho} = t$ and $\rho\bar{\rho} = n$. With these definitions, the Eichler-Selberg Trace Formula can be stated as follows:

\begin{TraceFormula}
If $k \geq 4$ is an even integer and $n \in \Z^+$, then the trace of the Hecke operator $T(n)$ on the space of cusp forms $S_k$ is given by 
\begin{equation}\label{eq: traceformula}
    \Tr T(n) = -\frac 12 \sum_{t= -\infty}^\infty P_k(t,n)H(4n-t^2) - \frac 12 \sum_{dd' = n}\min(d,d')^{k-1},
\end{equation}
where the second sum is taken over all factorizations $dd' = n$  with  $d,d' \geq 1$.
\end{TraceFormula}
\noindent Note that the first sum is in fact finite since $H(4n-t^2)$ is non-zero only when $\abs{t} \leq 2\sqrt{n}$.

\subsection{Proof of Theorem~\ref{thm: traceclassification}}\label{subsec: proofoftrace}
    To prove (i), we use the Eichler-Selberg trace formula to show that $\Tr_k(n) \equiv \Tr_{k+p(p^2-1)}(n) \pmod p$ for each $n \ge 1$.
    Using Fermat's little theorem, we see that $\min(d,d')^{(k+p(p^2-1))-1} \equiv \min(d,d')^{k-1} \pmod p$,
    hence we need only to consider the first term in \eqref{eq: traceformula}. In particular, it will suffice to prove that $$P_k(t,n) \equiv P_{k+p(p^2-1)}(t,n) \pmod p$$ for each $n \geq 1$ and $t \in \Z$ such that $\abs{t} \leq 2\sqrt{n}$.
    
    If $t^2 = 4n$, we have that $P_k(t,n)$ is the $x^{k-2}$ coefficient of
    $$(1-tx + nx^2)^{-1} = \Big(1 - \frac{tx}{2}\Big)^{-2} = \left(\sum_{i=0}^\infty \Big(\frac{tx}{2}\Big)^i \right)^2 = \sum_{j=0}^\infty \,\frac{j+1}{2^j}\, t^j x^j,$$
    i.e., $P_k(t,n) = (k-1)2^{2-k} t^{k-2}$. The desired congruence $P_k(t, n) \equiv P_{k+p(p^2-1)}(t, n) \pmod p$ then follows by Fermat's little theorem.
    
    From now on, we can assume that $t^2 - 4n < 0$. In particular, using the expression for $P_k(t,n)$ given in \eqref{eq: P}, we see that $\rho, \bar{\rho}$ lie in the imaginary quadratic field $K = \Q(\sqrt{t^2-4n}).$ We now split into cases based on how $p$ splits in $\mathcal{O}_K$. We will see that for fixed $t,n$, the values $P_k(t,n)$ mod $p$ are periodic in $k$ with periods $p^2-1$, $p-1$, or $p(p-1)$ depending on the case.

    \textbf{Case 1:} If $p$ is inert in $\mathcal{O}_K$, we have $\mathcal{O}_K/(p) \cong \F_{p^2}$, so in particular either $\rho \in (p)$ or $\rho^{p^2-1} \equiv 1 \in \mathcal{O}_K/(p)$. If $\rho \in (p)$, i.e., $v_p(\rho) \ge 1$, then the same is true of $\bar{\rho}$, and
    $$P_{a+1}(t,n) = \frac{\rho^a - \bar{\rho}^a}{\rho - \bar{\rho}} = \rho^{a-1} + \rho^{a-2} \bar{\rho} + \cdots + \bar{\rho}^{a-1}$$ 
    is of valuation $\ge a-1$, which means that $P_{k_1}(t,n) \equiv P_{k_2}(t,n) \equiv 0 \pmod p$ for any even $k_1,k_2 \ge 4$. If $\rho \notin (p)$, then the same is true of $\bar{\rho}$, so we have $\rho^{p^2-1}, \bar{\rho}^{p^2-1} \equiv 1 \in \mathcal{O}_K/(p)$. Hence in particular, for any $c \in \Z_{\geq 0}$ we have that
    $$\rho^{k+c(p^2-1)-1}-\bar{\rho}^{k+c(p^2-1)-1} \equiv \rho^{k-1} - \bar{\rho}^{k-1} \in \mathcal{O}_K/(p).$$
    Note that $\rho - \bar{\rho} = \sqrt{t^2-4n} \notin (p)$ by assumption that $p$ is not ramified. Therefore $p \nmid t^2 - 4n$, and so we can divide the above equation by $\rho - \bar{\rho}$ to get that $P_{k+c(p^2-1)}(t,n) \equiv P_k(t,n) \in \F_p$. Thus, for fixed $t,n$ so that $p$ is inert in $\Q(\sqrt{t^2-4n})$, we see that the values $P_k(t,n)$ mod $p$ are in fact $(p^2-1)$-periodic in $k$.
    
    \textbf{Case 2:} If $p$ splits in $\mathcal{O}_K$ as $(p) = \mathfrak{p}_1 \mathfrak{p}_2$ with $\mathfrak{p}_1\neq \mathfrak{p}_2$, we have $\mathcal{O}_K/(p) \cong \F_p \times \F_p$, so the $i$th coordinate of $\rho^{p-1} \in \F_p \times \F_p$ is $1$ if $\rho \notin \mathfrak{p}_i$, and 0 otherwise. If $\rho \in \mathfrak{p}_i$ for some $i$, then expanding $P_{a+1}(t,n)$ as above, we find that $P_{a+1}(t,n) \equiv \bar{\rho}^{a-1} \in \mathcal{O}_K/\mathfrak{p}_i$,
    and therefore $P_k(t,n) \equiv P_{k+c(p-1)}(t,n) \equiv 0 \pmod p$ for any even $k \ge 4$ and any $c \in \Z_{\ge 0}$.
    Otherwise, if $\rho, \bar{\rho} \notin \mathfrak{p}_1, \mathfrak{p}_2$, we have $\rho^{p-1}, \bar{\rho}^{p-1} \equiv 1 \pmod p$, which implies that $P_{k+c(p-1)}(t,n) \equiv P_k(t,n) \pmod p$ for any $c \in \Z_{\geq 0}$ by the same argument as above: we have $$\rho^{k+ c(p-1) - 1} - \bar{\rho}^{k + c(p-1) - 1} \equiv \rho^{k-1} - \bar{\rho}^{k-1} \pmod{\mathfrak{p}_i}$$
    for each $i$, and as before, since $t^2 - 4n \notin (p)$ (as $p$ does not ramify), $\rho - \bar{\rho} = \sqrt{t^2-4n} \notin \mathfrak{p}_i$. Dividing by $\rho - \bar{\rho}$, we get $P_{k+c(p-1)}(t,n) \equiv P_k(t,n) \pmod{\mathfrak{p}_i}$ for each $i$, and so also mod $p$. Therefore, for fixed $t,n$ so that $p$ splits in $\Q(\sqrt{t^2-4n})$, we see that the values $P_k(t,n)$ mod $p$ are $(p-1)$-periodic in $k$.
    
    \textbf{Case 3:} Finally, suppose that $p$ ramifies in $\mathcal{O}_K$, i.e., that $p = \mathfrak{p}^2$ for some prime $\mathfrak{p}$.
    If $v_{ \mathfrak{p}}(\rho) \ge 1$, then $v_{\mathfrak{p}}(\bar{\rho}) \ge 1$ as well, and so
    by the same argument as in Case 1, we find that $P_{k_1}(t,n) \equiv P_{k_2}(t,n) \equiv 0 \pmod p$ for all even integers $k_1,k_2 \ge 4$. Otherwise, we have that $\rho^{p-1}, \bar{\rho}^{p-1} \equiv 1 \in \mathcal{O}_K/\mathfrak{p}$.
    We note that in this case, $\rho - \bar{\rho} = \sqrt{t^2-4n} \in \mathfrak{p}$, as $t^2-4n \in (p)$ by assumption that $p \ne 2$ ramifies, so
    we must add a factor of $p$ in
    \begin{equation}\label{eq: case3ram}
    P_k(t,n) \equiv (k-1)\rho^{k-2} \equiv (k+cp(p-1)-1)\rho^{k+cp(p-1)-2} \equiv P_{k+cp(p-1)}(t,n) \pmod{\mathfrak{p}}.
    \end{equation}
    Thus in this case, for fixed $t,n$ so that $p$ ramifies in $\Q(\sqrt{t^2-4n})$, the values $P_k(t,n)$ mod $p$ are $p(p-1)$-periodic in $k$.
    
    We verify (ii) and (iii) computationally. First, note that for any pair of weights $k_2 > k_1 \ge 4$ satisfying (ii) or (iii), we have that $k_2 = k_1 + c(p - 1)$ for some $c \in \Z^+$. Since $E_{p-1} \equiv 1 \pmod{p}$ for all primes $p \geq 5$, we may embed $T_{k_1}$ into $M_{k_2}$ modulo $p$ as $E_{p-1}^c T_{k_1}$ without changing the congruence mod $p$. To compare two forms of common weight $k_2$ mod $p$, it suffices to consider the first $\lfloor \frac{k_2}{12} \rfloor$ coefficients of their $q$-expansions \cite[Theorem 2.58]{OnoCBMS}. That is, we verify that $\Tr_{k_2}(n) \equiv m \cdot \Tr_{k_1}(n) \pmod{p}$ with the appropriate choice of $m \in \F_p$ for all $1 \leq n \leq \lfloor \frac{k_2}{12} \rfloor$.
    Each Hecke trace is computed using the Eichler-Selberg trace formula \eqref{eq: traceformula}. Due to the periodicity established in (i), it suffices to consider pairs of weights of the form $4 \le k_1 < k_2 < p(p^2 - 1) + 4$. Since pairs of weights satisfying (ii) or (iii) exist only for $p \in \{5,7,11\}$, the proof thus consists of a finite check. For SageMath code used to carry out the verification necessary to prove (ii) and (iii), see \cite{code}. \qed
    
\begin{example}
As an example of the calculations required to prove Theorem~\ref{thm: traceclassification}, we verify that $T_{24}(z) \equiv T_{28}(z) \pmod{5}$, as claimed in (iii). Since $\Tr_{24}(1) = \dim S_{24} = 2 = \dim S_{28} = \Tr_{28}(1)$, it suffices to verify that $\Tr_{24}(2) \equiv \Tr_{28}(2) \pmod{5}$. Since $28 \equiv 24 \pmod{p - 1}$, we have cancellation in the second terms of \eqref{eq: traceformula} by Fermat's little theorem, so we need only to verify that
    \begin{equation}\label{eq: caseiiiii}
        \sum_{t=-2}^{2} P_{24}(t,2)H(8 - t^2) \equiv \sum_{t=-2}^{2} P_{28}(t,2)H(8 - t^2) \pmod{5}.
    \end{equation}

We compute the following values for $P_{24}(t,2), P_{28}(t,2)$, and $H(8-t^2)$:
    \begin{table}[H]
    \centering
    \begin{tabu}{|c|c|c|c|}
        \hline
        $t$ & $P_{24}(t,2)$ & $P_{28}(t,2)$ & $H(8 - t^2)$ \\ \hline
        $0$ & $-2048$ & $-8192$ & $1$ \\ \hline
        $\pm 1$ & $967$ & $8279$ & $1$ \\ \hline
        $\pm 2$ & $-2048$ & $8192$ & $1/2$ \\ \hline
    \end{tabu}
    \vspace{2mm}
    \caption{Example Calculations of $P_k(t,n)$ and $H(4n - t^2)$}
    \vspace{-6mm}
    \end{table}

\noindent Using these numerical values, we easily see that \eqref{eq: caseiiiii} indeed holds.
\end{example}

\subsection{Proof of Theorem~\ref{thm: hatclassification}}\label{subsec: proofofhat}

The proof of Theorem~\ref{thm: hatclassification} relies on the following congruence relations for the polynomials $P_k(t,n)$ mod $p$:

\begin{lemma}\label{lemma: Prelation}
Consider some fixed $(t,n) \in \Z \times \Z^+$ so that $t^2 \le 4n$, and any prime $p \geq 5$ such that $\gcd(p,n) = 1$. Let $k \ne 2$ be a non-negative even integer. If $p$ is unramified in $\Q(\sqrt{t^2-4n})$ and $m \in \Z^+$, then
\begin{equation*}
P_{k+m(p^2-1)}(t,n) \equiv \begin{dcases}
-n^{-1} \,\,\pmod p & \text{if } k = 0, \\
P_k(t,n) \,\,\pmod p & \text{otherwise}.
\end{dcases}
\end{equation*}
If $p$ is ramified in $\Q(\sqrt{t^2-4n})$ or $t^2 - 4n = 0$, then 
\begin{equation*}
P_{k+m(p^2-1)}(t,n)\equiv (k-m-1)n^{\frac{k-2}{2}} \pmod p.
\end{equation*}
\end{lemma}

\begin{proof}
Suppose first that $p$ is not ramified in $K:= \Q(\sqrt{t^2-4n})$. By the periodicity established in Cases 1 and 2 of the proof of Theorem \ref{thm: traceclassification}, we have $P_k(t,n) \equiv P_{k+m(p^2-1)}(t,n) \pmod{p}$ if $k \ge 4$, and $P_{m(p^2-1)}(t,n) \equiv P_{p^2-1}(t,n)$ if $k = 0$. For $k \ge 4$, this is the desired statement; for $k = 0$, we note that if $\mathfrak{p}$ is a prime of $\mathcal{O}_K$ above $p$, then $\rho - \bar{\rho} = \sqrt{t^2 - 4n} \notin \mathfrak{p}$ by the assumption that $p$ does not ramify in $K$. Also, $\rho, \bar{\rho} \notin \mathfrak{p}$ by assumption that $\gcd(p,n) = 1$.
Hence, we have that $\rho^{p^2-1}, \bar{\rho}^{p^2-1} \equiv 1 \in \mathcal{O}_K/\mathfrak{p}$, and so
$$P_{p^2-1}(t,n) = \frac{\rho^{p^2-2}-\bar{\rho}^{p^2-2}}{\rho - \bar{\rho}} \equiv \frac{\rho^{-1} - \bar{\rho}^{-1}}{\rho - \bar{\rho}} = \frac{-n^{-1}\sqrt{t^2-4n}}{\sqrt{t^2-4n}} = -n^{-1} \in \mathcal{O}_K/\mathfrak{p}.$$

Now, suppose that $p$ is ramified in $\Q(\sqrt{t^2-4n})$, so $p$ divides $t^2-4n$. Then with notation as before, we have that $\sqrt{t^2-4n} \in \mathfrak{p}$, and as we are assuming that $n \notin (p)$, we must have that $t \notin (p)$ as well. In particular, this means that $\rho, \bar{\rho} \equiv \frac{t}{2} \ne 0 \in \mathcal{O}_K/\mathfrak{p}$.
Consequently, it follows from \eqref{eq: case3ram} that \begin{align*}
P_{k+m(p^2-1)}(t,n) &\equiv (k+m(p^2-1)-1)\rho^{k+m(p^2-1)-2} \\
&\equiv (k-m-1)\rho^{k-2} \equiv (k-m-1)n^{\frac{k-2}{2}} \in \mathcal{O}_K/\mathfrak{p},
\end{align*} 
where we use that $\rho \equiv \frac{t}{2}$ and $(\frac{t}{2})^2 \equiv n \in \mathcal{O}_K/\mathfrak{p}$.
Similarly, if $t^2 - 4n = 0$, then the proof of Theorem \ref{thm: traceclassification} shows that
$$P_{k+m(p^2-1)}(t,n) = (k+m(p^2-1)-1) 2^{2-k-m(p^2-1)} t^{k+m(p^2-1)-2} \equiv (k-m-1)n^{\frac{k-2}{2}} \pmod{p}. $$
\end{proof}

Furthermore, we will need the following important property of the Hurwitz class numbers \cite{KroneckerHurwitz}:
\begin{thm*}[Kronecker-Hurwitz Relation]
For any positive integer $n$, we have that
\begin{equation}\label{eq: kroneckerhurwitz}
    \sum_{t = -\infty}^\infty H(4n-t^2) = \sum_{dd' = n} \max(d,d'),
\end{equation}
where the second sum is taken over all factorizations $dd' = n$ with $d,d' \geq 1$.
\end{thm*}

\newpage
\begin{proof}[Proof of Theorem~\ref{thm: hatclassification}]
    We use the Eichler-Selberg trace formula \eqref{eq: traceformula} to show that $$m \cdot \Tr_{k+(p^2-1)}(n) \equiv \Tr_{k+m(p^2-1)}(n) \pmod p$$ for $k \in \{0,4,6,8,10,14\}$ and $n \geq 1$ such that $\gcd(p,n) = 1$.
    The proof of Theorem~\ref{thm: traceclassification}~(i) shows that $P_{k+(p^2-1)}(t,n) \equiv P_{k+m(p^2-1)}(t,n) \pmod p$ whenever $p$ is unramified in $\Q(\sqrt{t^2-4n})$. Thus, denoting by $S$ (resp.~$T$) the set of $t \in \Z$ such that $t^2 \leq 4n$ and $p$ is ramified (resp.~unramified) in $\Q(\sqrt{t^2-4n})$ (where $t \in S$ if $t^2 = 4n$), we need to show that
    \begin{align}\label{eq: kcases}
    \begin{split}
    \sum_{t \in S}&\Big(m \, P_{k+p^2-1}(t,n)-P_{k+m(p^2-1)}(t,n)\Big)H(4n-t^2)\\
    &+\sum_{t \in T} (m-1)P_{k+p^2-1}(t,n)H(4n-t^2)+\sum_{dd' = n}(m-1)\min(d,d')^{k+p^2-2} \equiv 0 \pmod p,
    \end{split}
    \end{align}
    using that we have $P_{k+p^2-1}(t,n) \equiv P_{k+m(p^2-1)}(t,n) \pmod{p}$ for $t \in T$ by Lemma~\ref{lemma: Prelation}, and that $\min(d,d')^{k+p^2-2} \equiv \min(d,d')^{k+m(p^2-1)-1} \pmod{p}.$
    
    Suppose first that $k \neq 0$. We note that for $t \in S$, by Lemma~\ref{lemma: Prelation} we have
    $$m\, P_{k+p^2-1}(t,n) - P_{k+m(p^2-1)}(t,n) \equiv m(k-2)n^{\frac{k-2}{2}} - (k-m-1)n^{\frac{k-2}{2}}$$
    $$= (m-1)(k-1) n^{\frac{k-2}{2}} \equiv (m-1)P_k(t,n) \pmod p.$$
    Similarly, for $t \in T$, Lemma~\ref{lemma: Prelation} shows that $P_{k+p^2-1}(t,n) \equiv P_k(t,n)$. Moreover, since $d,d' \not\equiv 0 \pmod p$, we have
    $\min(d,d')^{k+p^2-2} \equiv \min(d,d')^{k-1} \pmod p$.
    Thus \eqref{eq: kcases} simplifies to 
    \begin{equation*}
    (m-1)\Big(\sum_{t = -\infty}^\infty P_k(t,n) H(4n-t^2)+\sum_{dd' = n} \min(d,d')^{k-1}\Big) \equiv -2(m-1)\Tr_k(n) \equiv 0 \pmod p,
    \end{equation*}
    where the last relation holds since $S_k = 0$ and therefore
    $\Tr_k(n) = 0$ for each $k \in \{4,6,8,10,14\}$.
    
    Now, consider the case in which $k = 0$. By Lemma~\ref{lemma: Prelation}, we have that $$P_{m(p^2-1)}(t,n) \equiv -n^{-1} \pmod p$$ whenever $t \in T$, and that
    $$mP_{p^2-1}(t,n) - P_{m(p^2-1)}(t,n) \equiv m \cdot (-2)n^{-1} - (-m-1) n^{-1} = (1-m) n^{-1} \pmod{p}$$
    for $t \in S$. Hence \eqref{eq: kcases} simplifies to
    \begin{align*}
    \sum_{t \in S}(m-1)(-n^{-1})H(4n-t^2)+\sum_{t \in T}& (m-1)(-n^{-1})H(4n-t^2)\\
    &+\sum_{dd' = n}(m-1)\min(d,d')^{p^2-2} \equiv 0 \pmod p,
    \end{align*}
    which is equivalent to the statement that
    \begin{align*}
    (m-1)\Big(-n^{-1}\sum_{t = -\infty}^\infty H(4n-t^2)+\sum_{dd' = n} \min(d,d')^{p^2-2}\Big) \equiv 0 \pmod p.
    \end{align*}
    By the Kronecker-Hurwitz relation \eqref{eq: kroneckerhurwitz}, we can write the above as
    \begin{equation*}
    (m-1)\Big(-n^{-1}\sum_{dd' = n}\max(d,d') + \sum_{dd' = n} \min(d,d')^{p^2-2}\Big) \equiv 0 \pmod p.
    \end{equation*}
    For any $d,d'$ such that $dd' = n$, we have that $\min(d,d')^{p^2-2} \equiv \min(d,d')^{-1} \pmod p$, since $\min(d,d') \not\equiv 0 \pmod p$. Moreover, if we assume without loss of generality that $d  \leq d'$, then $-n^{-1}\max(d,d') + \min(d,d')^{p^2-2} = -n^{-1}d'+d^{p^2-2} \equiv -d^{-1}+d^{-1} \equiv 0 \pmod p.$ This shows that $\big(-n^{-1}\sum_{dd' = n}\max(d,d') + \sum_{dd' = n} \min(d,d')^{p^2-2}\big)$ vanishes term-wise mod $p$, and so we conclude that \eqref{eq: kcases} holds for $k = 0$.
\end{proof}


\section{Divisor Polynomials mod \texorpdfstring{$p$}{p}} \label{sec: modformsmodp}
Recall that the space $M_k$ of modular forms of weight $k$ on $\SL_2(\Z)$ is generated over $\C$ by the monomials $E_4^\alpha E_6^\beta$ such that $4\alpha + 6 \beta = k$ and $\alpha,\beta \in \Z_{\geq 0}$. Given $\Delta(z) := \frac{E_4(z)^3 - E_6(z)^2}{1728}$ and $j(z) := \frac{E_4(z)^3}{\Delta(z)}$, this means that every modular form $f \in M_k$ can be decomposed as in \eqref{eq: fdecomposition},
where $m = \dim S_k$, the divisor polynomial $F(f;x) \in \C[x]$ is a polynomial of degree $m$, and
\begin{equation}\label{eq: epdelta}
\delta :=
\begin{cases}
0 & k \equiv 0 \pmod 6,\\
1 & k \equiv 4 \pmod 6,\\
2 & k \equiv 2 \pmod 6,
\end{cases} \quad 
\textrm{ and } \quad
\ep :=
\begin{cases}
0 & k \equiv 0 \pmod 4,\\
1 & k \equiv 2 \pmod 4.
\end{cases}
\end{equation}
The divisor polynomial of a product of modular forms can be described as follows:

\begin{prop}\label{prop: table}
If $f$ and $g$ are modular forms of even weights $k_1,k_2 \ge 4$, respectively, then
$$F(fg;x) = x^{a_{k_1,k_2}} (x-1728)^{b_{k_1,k_2}} F(f;x) F(g;x),$$
where the values of $(a_{k_1,k_2},b_{k_1,k_2})$ are given in the following table:
\vspace{2mm}
\begin{table}[H]
    \centering
\begin{tabu}{|c|c|c|c|c|c|c|}
    \hline 
     $\pmod{12}$ & $k_1 \equiv 0$ & $k_1 \equiv 2$ & $k_1 \equiv 4$ & $k_1 \equiv 6$ & $k_1 \equiv 8$ & $k_1 \equiv 10$\\ \hline
     $k_2 \equiv 0$ & $(0,0)$ & $(0,0)$ & $(0,0)$ & $(0,0)$ & $(0,0)$ & $(0,0)$\\ \hline
     $k_2 \equiv 2$ & $(0,0)$ & $(1,1)$ & $(1,0)$ & $(0,1)$ & $(1,0)$ & $(1,1)$\\ \hline
     $k_2 \equiv 4$ & $(0,0)$ & $(1,0)$ & $(0,0)$ & $(0,0)$ & $(1,0)$ & $(0,0)$\\ \hline
     $k_2 \equiv 6$ & $(0,0)$ & $(0,1)$ & $(0,0)$ & $(0,1)$ & $(0,0)$ & $(0,1)$\\ \hline
     $k_2 \equiv 8$ & $(0,0)$ & $(1,0)$ & $(1,0)$ & $(0,0)$ & $(1,0)$ & $(1,0)$\\ \hline
     $k_2 \equiv 10$ & $(0,0)$ & $(1,1)$ & $(0,0)$ & $(0,1)$ & $(1,0)$ & $(0,1)$\\ \hline
\end{tabu}
\vspace{2mm}
\caption{Values of $(a_{k_1,k_2},b_{k_1,k_2})$ for each $k_1,k_2 \pmod{12}$.}
\label{tbl: ab}
\end{table}
\end{prop}
\vspace{-0.5cm}

\begin{proof}
    Each entry of the table is computed in the same way, so we provide only an illustrative example. Suppose that $f$ is a modular form of weight $k_1 \equiv 2 \pmod{12}$ and $g$ is a modular form of weight $k_2 \equiv 10 \pmod{12}$. Then we have
    $$f(z) = \Delta(z)^{\frac{k_1-14}{12}} E_4(z)^2 E_6(z) F(f;j(z)),$$
    $$g(z) = \Delta(z)^{\frac{k_2-10}{12}} E_4(z) E_6(z) F(g;j(z)).$$
    Using that $f(z)g(z)$ is modular of weight $k_1 + k_2 \equiv 0 \pmod{12}$, it follows that
    $$f(z)g(z) = \Delta(z)^{\frac{k_1+k_2}{12}} F(fg;j(z)).$$
    The above equations then give
    \begin{align*}
    \Delta^{\frac{k_1+k_2-24}{12}} E_4^3 E_6^2 F(f;j(z)) F(g;j(z)) &= \Delta^{\frac{k_1+k_2}{12}} F(fg;j(z)) \\
    \implies F(fg;j(z))= \Delta^{-2} E_4^3 E_6^2 F(f;j(z)) F(g;j(z)) &= j(z)(j(z)-1728) F(f;j(z)) F(g;j(z)),
    \end{align*}
    where we are using that $E_4(z)^3 = \Delta(z) j(z)$ and $E_6(z)^2 = \Delta(z)(j(z)-1728)$.
\end{proof}

We now turn to reductions of modular forms mod $p$. If $f(z)$ is a modular form whose $q$-expansion has $p$-integral coefficients, we can reduce all coefficients of $f, \Delta, E_4$, and $E_6$ mod $p$ and
compute the divisor polynomial as before. We recall that 
$E_{p-1} \equiv 1 \in \F_p[[q]]$, and that 
if $f$ and $g$ are modular forms of even weights $k_1,k_2 \ge 4$, respectively, such that $f(z) \equiv g(z) \not\equiv 0 \in \F_p[[q]]$, then $k_2 - k_1 \equiv 0 \pmod{p-1}$ \cite[Proposition 2.43]{OnoCBMS}. In particular, if $k_2 \ge k_1$ and $n= \frac{k_2-k_1}{p-1}$, then $g$ and $E_{p-1}^n f$ are modular forms of weight $k_2$ with the same reduction mod $p$, and hence the same divisor polynomial mod $p$. Recalling from \cite{K-Z} that $F(E_{p-1};x) \equiv \widetilde{S}_p(x) \in \F_p[x]$, we can describe the relationship between the divisor polynomials of $f$ and $g$ mod $p$ by the following propositions:

\begin{prop}\label{prop: epn}
For $p \geq 5$ prime and $n \geq 1$, we have
    \begin{equation*}
    F(E_{p-1}^n;x) \equiv F(E_{p-1};x)^n x^{\delta_p \lfloor n/3 \rfloor} (x-1728)^{\ep_p \lfloor n/2 \rfloor} \in \F_p[x],
    \end{equation*}
    where $\delta_p$ and $\ep_p$ are as in \eqref{def: deltaepsilon}.
\end{prop}

\begin{proof}
    This follows by induction on $n$ from the formula
    \begin{equation*}
        F(E_{p-1}^n;x) \equiv F(E_{p-1};x) F(E_{p-1}^{n-1};x) x^{a_{p-1, (n-1)(p-1)}} (x-1728)^{b_{p-1, (n-1)(p-1)}} \in \F_p[x],
    \end{equation*}
    where the values of $(a_{p-1, (n-1)(p-1)},b_{p-1,(n-1)(p-1)})$ are given in Table~\ref{tbl: ab}.
\end{proof}

\begin{proposition}\label{prop: divpolycong}
Let $p \geq 5$ be prime, $n$ be a positive integer, and $f \in M_k$ and $g \in M_{k+n(p-1)}$ be such that $f(z) \equiv g(z) \in \F_p[[q]]$. Then
    \begin{equation*}
        F(g;x) \equiv F(f;x)
        \widetilde{S}_p(x)^n x^{\delta_p \lfloor n/3 \rfloor + a} (x - 1728)^{\ep_p \lfloor n/2 \rfloor + b} \in \F_p[x],
    \end{equation*}
    where $\delta_p$ and $\ep_p$ are given in \eqref{def: deltaepsilon}, and $a,b$ are given by $(a_{k, n(p-1)}, b = b_{k,n(p-1)})$ in Table~\ref{tbl: ab}.
\end{proposition}

\begin{proof}
    As noted above, we have that
    $$F(g;x) \equiv F(fE_{p-1}^n;x) \in \F_p[x],$$
    so by Propositions~\ref{prop: table} and \ref{prop: epn}, it follows that
    \begin{align*}
        F(g;x) \equiv F(fE_{p-1}^n; x)
        &\equiv F(f;x) F(E_{p-1}^n;x) x^{a_{k, n(p-1)}} (x-1728)^{b_{k,n(p-1)}} \\
        &\equiv F(f;x) F(E_{p-1};x)^n x^{\delta_p \lfloor n/3 \rfloor + a_{k,n(p-1)}} (x - 1728)^{\ep_p \lfloor n/2 \rfloor + b_{k,n(p-1)}} \in \F_p[x].
    \end{align*}
\end{proof}

\newpage
\section{Proofs of Theorems~\ref{thm: congruence} and \ref{thm: hatcongruence}}\label{sec: mainproofs}
We use the results of Sections~\ref{sec: congruences} and \ref{sec: modformsmodp} to prove more general versions of Theorems~\ref{thm: congruence} and \ref{thm: hatcongruence}.

\subsection{Divisor Polynomials of Trace Forms}\label{subsec: thm4.1}
Theorem~\ref{thm: congruence} and Corollary~\ref{cor: sspmod} are immediate consequences of the following more general result, which uses Theorem~\ref{thm: traceclassification} and Proposition~\ref{prop: divpolycong}.
\begin{theorem}\label{thm: preccong}
    If $p \geq 5$ is prime, $k \geq 4$ is an even integer, and 
    $$n = \begin{cases}
    (p^2+p) \lfloor \frac{k}{p^3-p} \rfloor & k \not\equiv 0,2 \pmod{p^3-p},\\
    (p^2+p) (\lfloor \frac{k}{p^3-p} \rfloor - 1) & k \equiv 0,2 \pmod{p^3-p},
    \end{cases}
    $$
    then
    \begin{equation*}
        F(T_k;x) \equiv F(T_{k-n(p-1)};x) \widetilde{S}_p(x)^n x^{\delta_p \lfloor n/3 \rfloor + a} (x - 1728)^{\ep_p \lfloor n/2 \rfloor + b} \pmod{p},
    \end{equation*}
    where $\delta_p$ and $\ep_p$ are as defined in \eqref{def: deltaepsilon}, and $a,b \in \{0,1\}$ are given by $a_{k, n(p-1)}, b_{k, n(p-1)}$ in Proposition~\ref{prop: table}.
    In particular, $\widetilde{S}_p(x)^n$ divides $F(T_k;x)$ mod $p$.
\end{theorem}

\begin{proof}
    With $p, k, n$ as in the statement of the theorem, let
    $k' = k - n(p-1)$. Then $k' \ge 4$ is an even integer
    satisfying $k' \equiv k \pmod{p^3-p}$. Hence by Theorem~\ref{thm: traceclassification}, we have that $T_k(z) \equiv T_{k'}(z) \pmod{p}$. Now, by Proposition~\ref{prop: divpolycong}, 
    $$F(T_k;x) \equiv F(T_{k'};x) \widetilde{S}_p(x)^n x^{\delta_p \lfloor n/3 \rfloor + a} (x - 1728)^{\ep_p \lfloor n/2 \rfloor + b} \in \F_p[x],$$
    where $(a,b) = (a_{k', n(p-1)}, b_{k', n(p-1)}) =(a_{k, n(p-1)}, b_{k, n(p-1)})$ since $k' \equiv k \bmod{p^3-p}$, and therefore also mod $12$.
\end{proof}

\begin{cor}\label{cor: alphabeta}
    If $p \ge 5$ is prime, $k$ is a positive integer such that $k \equiv 12,16,18,20,22,26 \pmod{p^3-p}$, and $n = (p^2+p) \lfloor \frac{k}{p^3-p} \rfloor$, then
    $$F(T_k;x) \equiv \widetilde{S}_p(x)^n x^{\delta_p \lfloor n/3 \rfloor + a} (x - 1728)^{\ep_p \lfloor n/2 \rfloor + b} \pmod{p},$$
    where $\delta_p$ and $\ep_p$ are as in \eqref{def: deltaepsilon}, and $a,b$ are given by $a_{k, n(p-1)}, b_{k, n(p-1)}$ in Proposition~\ref{prop: table}.
\end{cor}

\begin{proof}
    By assumption, $k-n(p-1) \in \{12,16,18,20,22,26\}$, so in particular, $F(T_{k-n(p-1)};x) = 1$, as $T_{k-n(p-1)}(z)$ must be a constant multiple of the generator $\Delta(z) E_4(z)^\delta E_6(z)^\ep$ of $S_k$, with $\delta, \ep$ as defined in \eqref{eq: epdelta}. By the normalization of the $q$-coefficient $\Tr_{k-n(p-1)}(1) = 1$, it follows that this constant must be 1. The desired result now follows immediately from Theorem~\ref{thm: preccong}.
\end{proof}

\subsection{Divisor Polynomials of Modified Trace Forms}\label{subsec: thm4.3}
Theorem~\ref{thm: hatcongruence} is a consequence of the following more general result, which uses Theorem~\ref{thm: hatclassification} and Proposition~\ref{prop: divpolycong}.
\begin{theorem}\label{thm: prechatcong}
    If $p \geq 5$ is prime and $k \geq p^2-1$ is an even integer such that $k \equiv 0,4,6,8,10,14 \pmod{p^2-1}$, then
    \begin{equation*}
        F(\widehat{T}_k^{(p)};x) \equiv m \cdot F(\widehat{T}_{k-n(p-1)}^{(p)};x) \widetilde{S}_p(x)^n x^{\delta_p \lfloor n/3 \rfloor + a} (x - 1728)^{\ep_p \lfloor n/2 \rfloor + b} \pmod{p},
    \end{equation*}
    where $m = \lfloor \frac{k}{p^2-1} \rfloor$, $n = (p+1)(m-1)$, and $a,b \in \{0,1\}$ are again given by $a_{k,n(p-1)}, b_{k,n(p-1)}$ in Proposition~\ref{prop: table}. In particular, $\widetilde{S}_p(x)^n$ divides $F(\widehat{T}_k^{(p)};x)$.
\end{theorem}

\begin{proof}
    Let $p, k, m$ be as in the statement of the theorem.
    Theorem~\ref{thm: hatclassification} gives that
    $$\widehat{T}_k^{(p)}(z) = m \cdot \widehat{T}_{k-(m-1)(p^2-1)}^{(p)}(z).$$
    As in the remark following Theorem~\ref{thm: hatclassification}, we note that $\widehat{T}_k^{(p)}(z) \equiv \Theta^{p-1} T_k(z) \pmod p$. Since for $f \in M_k$ with integral coefficients, $\Theta(f)$ is the mod $p$ reduction of a modular form of weight $k+p+1$, it follows that $\widehat{T}_k^{(p)}$ is the mod $p$ reduction of a modular form of weight $k + p^2 - 1$  \cite[Proposition 2.44]{OnoCBMS}. Similarly, $\widehat{T}_{k-(m-1)(p^2-1)}^{(p)}$ is the reduction of a form of weight $k-(m-2)(p^2-1)$.
    Using these weights to compute divisor polynomials for $\widehat{T}_k^{(p)}(z)$ and $\widehat{T}_{k-(m-1)(p^2-1)}^{(p)}(z)$, we apply Proposition~\ref{prop: divpolycong} to get
    $$F(\widehat{T}_k^{(p)};x) \equiv m \cdot F(\widehat{T}_{k-n(p-1)}^{(p)};x) \widetilde{S}_p(x)^n x^{\delta_p \lfloor n/3 \rfloor + a} (x - 1728)^{\ep_p \lfloor n/2 \rfloor + b} \in \F_p[x],$$
    where $n = (m-1)(p+1)$, and $(a,b) = (a_{k-(m-2)(p^2-1), n(p-1)}, b_{k-(m-2)(p^2-1), n(p-1)}) = (a_{k,n(p-1)}, b_{k, n(p-1)})$ since $p^2-1 \equiv 0 \pmod{12}$.
\end{proof}

\section{Examples}\label{sec: numerics}
We conclude by giving a number of examples of Proposition~\ref{prop: divpolycong} as well as Theorems~\ref{thm: preccong} and \ref{thm: prechatcong}. These examples were computed using the SageMath code given in \cite{code}.
\begin{example}
    Let $p = 23$ and $k = 12172$. Note that $k = 28+ p(p^2 - 1)$, hence $k$ and 28 satisfy the condition of Theorem~\ref{thm: traceclassification} (i). We have that
    \begin{align*}
        T_{12172}(z) &= 1014q + 625630\dots201640q^2 + \dots \\
         &\equiv 2q + 18q^3 + 9q^4 + 18q^5 + 9q^6 + \dots \pmod{23}
    \end{align*}
    where the coefficient of $q^2$ is $\approx 6.3 \times 10^{1831}$. The divisor polynomial of $T_{12172}(z)$ is then
    \begin{align*}
        F(T_{12172}; x) \equiv 2(x + 7)(x + 4)^{552} x^{184}  (x + 20)^{276} \pmod{23},
    \end{align*}
    while
    \begin{equation*}
        F(T_{28}; x) \equiv 2(x + 7) \pmod{23}.
    \end{equation*}
    Therefore,
    \begin{equation*}
        F(T_{12172}; x) \equiv F(T_{28}; x)\widetilde{S}_{23}(x)^n x^{\lfloor n/3 \rfloor} (x -1728)^{\lfloor n/2 \rfloor} \pmod{23}
    \end{equation*}
    for $n = (23^2 + 23)\lfloor \frac{12172}{23^3 - 23} \rfloor = 552$, as indicated by Theorem~\ref{thm: preccong}.
\end{example}

\begin{example}
    Let $p = 5$, $k_1 = 28 = 4 + (p^2 - 1)$, and $k_2 = 76 = 4+ 3(p^2 - 1)$. Then $k_1$ and $k_2$ satisfy the conditions of Theorem~\ref{thm: traceclassification} (ii). Using the Eichler-Selberg trace formula \eqref{eq: traceformula}, we find that
    \begin{align*}
        3\cdot T_{28}(z) &= 6q - 24840q^2 - 3858840q^3 + \dots, \\
        T_{76}(z) &= 6q - 57080822040q^2 - 785092363818710040q^3 + \dots.
    \end{align*}
    As Theorem~\ref{thm: traceclassification} indicates, we obtain that
    \begin{equation*}
        3\cdot T_{28}(z) \equiv T_{76}(z) \equiv q + 3q^4 + 2q^6 + 2q^9 + \dots \pmod{5}.
    \end{equation*}
    Moreover, the divisor polynomials $F(T_{k_1};x)$ and $F(T_{k_2};x)$ are such that
    \begin{align*}
        3 \cdot F(T_{28}; x) &\equiv \phantom{x^4}(x + 4) \pmod{5}, \\
        F(T_{76}; x) &\equiv x^4(x +4) \pmod{5}.
    \end{align*}
    Therefore, we find that 
    \begin{align*}
        F(T_{76}; x) \equiv 3\cdot F(T_{28}; x) \widetilde{S}_5(x)^n x^{\lfloor n/3 \rfloor}\pmod{5}
    \end{align*}
    for $n = \frac{78 - 26}{5 - 1} = 12$, which is in agreement with Proposition~\ref{prop: divpolycong}.
\end{example}

\begin{example}\label{ex: three}
    Let $p = 17$ and $k = 582$. Note that $k = 6+2(p^2 - 1) $, hence $k$ satisfies the conditions of Theorem~\ref{thm: prechatcong}. We have that
    \begin{equation*}
        \widehat{T}^{(17)}_{582}(z) = 14q + 3q^4 + 12q^9 + 13q^{13} + 16q^{15} + \dots.
    \end{equation*}
    The divisor polynomial of $\widehat{T}^{(17)}_{582}(z)$ is then
    \begin{equation*}
        F(\widehat{T}^{(17)}_{582}; x) \equiv 14 \cdot x^{14} (x + 9)^{42}Q(x) \pmod{17},
    \end{equation*}
    while, for $6+p^2-1 = 294$,
    \begin{equation*}
        F(\widehat{T}^{(17)}_{294}; x) \equiv 7 \cdot x^8 (x + 9)^{24}Q(x) \pmod{17},
    \end{equation*}
    where $Q(x) \in \F_p[x]$ is the irreducible polynomial
    \begin{equation*}
        Q(x) \equiv x^{15} + 13x^{14} + 10x^{13} + 13x^{12} + 4x^{11} + 8x^{10} + 13x^9 + 14x^8 + 15x^7 + 7x^6 + 4x^3 + 13x^2 + 8 \pmod{17}.
    \end{equation*}
    Thus
    \begin{equation*}
        F(\widehat{T}^{(17)}_{582}; x) \equiv 2 \cdot F(\widehat{T}^{(17)}_{294}; x) \widetilde{S}_{17}(x)^n x^{\lfloor n/3 \rfloor} \pmod{17}
    \end{equation*}
    for $n = (17 + 1)(\lfloor \frac{582}{17^2 - 1} \rfloor - 1) = 18$,
    as indicated by Theorem~\ref{thm: prechatcong}.
\end{example}

\bibliography{references}

\providecommand{\bysame}{\leavevmode\hbox to3em{\hrulefill}\thinspace}
\providecommand{\MR}{\relax\ifhmode\unskip\space\fi MR }
\providecommand{\MRhref}[2]{%
  \href{http://www.ams.org/mathscinet-getitem?mr=#1}{#2}
}
\providecommand{\href}[2]{#2}
\begin{thebibliography}{1}

\bibitem{Ahlgren}
S.~Ahlgren and K.~Ono, \emph{Weierstrass points on {$X_0(p)$} and supersingular
  {$j$}-invariants}, Math. Ann. \textbf{325} (2003), no.~2, 355--368.

\bibitem{KroneckerHurwitz}
J.~Gierster, \emph{Über {R}elationen zwischen {K}lassenzahlen bin\"{a}rer
  quadratischer {F}ormen von negativer {D}eterminante}, Math. Ann. \textbf{21}
  (1883), no.~1, 1--50.

\bibitem{code}
K.~Gomez, \emph{Supersingular loci from traces of hecke operators},
  \url{https://gist.github.com/kg583/ed0911aea395690c7dbe87e14ad55616}, 2021.

\bibitem{K-Z}
M.~Kaneko and D.~Zagier, \emph{Supersingular {$j$}-invariants, hypergeometric
  series, and {A}tkin's orthogonal polynomials}, Computational perspectives on
  number theory ({C}hicago, {IL}, 1995), AMS/IP Stud. Adv. Math., vol.~7, Amer.
  Math. Soc., Providence, RI, 1998, pp.~97--126.

\bibitem{ZagierTraceFormula}
S.~Lang, \emph{Introduction to modular forms}, Grundlehren der Mathematischen
  Wissenschaften, No. 222, Springer-Verlag, Berlin-New York, 1976.

\bibitem{OnoCBMS}
K.~Ono, \emph{The web of modularity: arithmetic of the coefficients of modular
  forms and {$q$}-series}, CBMS Regional Conference Series in Mathematics, vol.
  102, Published for the Conference Board of the Mathematical Sciences,
  Washington, DC; by the American Mathematical Society, Providence, RI, 2004.

\bibitem{Serre}
J.-P. Serre, \emph{Congruences et formes modulaires [d'apr\`es {H}. {P}. {F}.
  {S}winnerton-{D}yer]}, S\'{e}minaire {B}ourbaki, 24e ann\'{e}e (1971/1972),
  {E}xp. {N}o. 416, 1973, pp.~319--338. Lecture Notes in Math., Vol. 317.

\bibitem{Silverman}
J.~H. Silverman, \emph{The arithmetic of elliptic curves}, Graduate Texts in
  Mathematics, vol. 106, Springer-Verlag, New York, 1986.

\end{thebibliography}
\bibliographystyle{amsplain}

\end{document}